\documentclass{amsart}
\usepackage{graphicx}
\usepackage{amsmath, amsthm,  amsfonts, amscd, amssymb}
\usepackage{url}
\usepackage{mathrsfs}
\usepackage[all]{xy}
\usepackage{hyperref}

\DeclareMathOperator{\hol}{hol}

\DeclareMathOperator{\ad}{ad}
\DeclareMathOperator{\ev}{ev}

\DeclareMathOperator{\Diff}{Diff}
\DeclareMathOperator{\End}{End}

\theoremstyle{plain}
\newtheorem{theorem}{Theorem}[section]

\newtheorem{proposition}[theorem]{Proposition}

\theoremstyle{definition}
\newtheorem{definition}[theorem]{Definition}

\theoremstyle{remark}

\newtheorem{note}{Note}[section]
\newtheorem{remark}{Remark}[section]

\numberwithin{figure}{section}

\newcommand{\fg}{{\mathfrak g}}

\newcommand{\RR}{{\mathbb R}}

\renewcommand{\a}{\alpha}

\newcommand{\<}{\langle}
\renewcommand{\>}{\rangle}

\newcommand{\wt}{\widetilde}
\begin{document}

\title[String classes and equivariant cohomology]{Loop groups, string classes and equivariant cohomology}
\author{Raymond F. Vozzo}
\address{School of Mathematical Sciences\\
  University of Adelaide\\
  Adelaide, SA 5005 \\
  Australia}
  \email{raymond.vozzo@adelaide.edu.au}

\date{\today}

\begin{abstract}
We give a classifying theory for $LG$-bundles, where $LG$ is the loop group of a compact Lie group $G$, and present a calculation for the string class of the universal $LG$-bundle. We show that this class is in fact an equivariant cohomology class and give an equivariant differential form representing it. 
We then use the caloron correspondence to define (higher) characteristic classes for $LG$-bundles and to prove for the free loop group an analogue of the result for characteristic classes for based loop groups in \cite{Murray-Vozzo1}. These classes have a natural interpretation in equivariant cohomology and we give equivariant differential form representatives for the universal case in all odd dimensions.
\end{abstract}

\thanks{The author acknowledges the support of the Australian Research Council and useful discussions with Michael Murray and Mathai Varghese.}

\subjclass[2010]{55R10, 55R35, 57R20, 81T30}

\maketitle

\tableofcontents

\section{Introduction}

Let $LG$ be the space of smooth maps from the circle into a compact Lie group $G$. Then $LG$ is a (Frech\'et) Lie group. In this article we shall be considering characteristic classes of principal bundles whose structure group is $LG$. Recall that in \cite{Murray-Vozzo1} we constructed characteristic classes for bundles whose structure group is the \emph{based} loop group, $\Omega G$, using the caloron correspondence---a correspondence between loop group bundles and certain $G$-bundles. Part of what made that construction simple was the fact that there exists a nice model for the universal bundle and hence a relatively easy classifying theory. Here we would like to extend these results to the group of free loops, where the classifying theory is more complicated.

This article is organised as follows: In section 2 we review the formula for the string class from \cite{Murray:2003} and the extension of this class to higher dimensions (for the based loop group) \cite{Murray-Vozzo1}. In section 3 we outline the construction of the universal $LG$-bundle and give a classifying map for any $LG$-bundle. We show that characteristic classes for $LG$-bundles are given by equivariant cohomology classes on $G$. Section 4 reviews the necessary background on equivariant cohomology, in particular the Cartan model of equivariant differential forms and the Mathai-Quillen isomorphism. Sections 5 and 6 contain our main results concerning string classes: In section 5 we calculate the universal (degree three) string class and in section 6 we extend this to any odd degree, proving a theorem analogous to the main result in \cite{Murray-Vozzo1} for $\Omega G$ (Theorem \ref{T:LG Chern-Weil}).

\section{String classes for loop group bundles}

\subsection{The string class}

String structures were introduced by Killingback as the string theory analogue of spin structures \cite{Killingback:1987}. Suppose we have an $LG$-bundle $P \to M$. Since $LG$ has a central extension by the circle (see, for example, \cite{Pressley-Segal} for details) we can consider the problem of lifting the structure group of $P$ to the central extension $\widehat{L G}$ of $L G$. Physically, this is related to the problem of defining a Dirac-Ramond operator in string theory. Mathematically, one has an obstruction to doing this---a certain degree three cohomology class on the base of the bundle. This class is called the \emph{string class} of the bundle and we write $s(P) \in H^3(M)$. In \cite{Murray:2003} Murray and Stevenson give a formula for a de Rham representative of this class which is given by:

\begin{theorem}[\cite{Murray:2003}]\label{T:string class}
Let $P \to M$ be a principal $L G$-bundle. Let $A$ be a connection on $P$ with curvature $F$ and let $\Phi$ be a Higgs field for $P$. Then the string class of $P$ is 
represented in de Rham cohomology by the form
$$
-\frac{1}{4\pi^2} \int_{S^1} \< \nabla \Phi, F \> \, d\theta,
$$
where $\<\,\, ,\, \>$ is an invariant inner product on $\fg$ normalised so the longest root has length squared equal to $2$ and $\nabla \Phi = d \Phi + [A, \Phi] - \dfrac{\partial A}{\partial \theta}$.
\end{theorem}

The \emph{Higgs field} for $P$ in Theorem \ref{T:string class} is a map $\Phi \colon P \to L\fg$ satisfying $\Phi(p\gamma) = \ad(\gamma^{-1}) \Phi(p) + \gamma^{-1} \partial \gamma$ (where $\partial \gamma$ is the derivative in the loop direction, $\partial \gamma / \partial \theta$). A geometric interpretation of this map is given by the \emph{caloron correspondence} \cite{Murray:2003} (See also \cite{Murray-Vozzo1, Murray-Vozzo2}), which allows us to extend the definition of the string class to higher degrees. We shall discuss this correspondence in section \ref{SS:higher Omega G classes}.

\subsection{The universal string class for $\Omega G$-bundles}\label{SS:s(PG)}

Theorem \ref{T:string class} carries over word-for-word for the based loop group $\Omega G$. In this case there is a simple model for the universal bundle which we shall now present, and which makes it possible to easily calculate the universal string class.

Let $PG$ be the space of paths in $G$, that is smooth maps  $p \colon \RR \to G$ such that $p(0) =0$ and $p^{-1} \partial p$ is $2\pi$-periodic. Then this is acted on by $\Omega G$ and
$$
\xymatrix@C=4ex{\Omega G\ar[r]	& PG\ar[d]\\
					&G}
$$
is an $\Omega G$-bundle called the \emph{path fibration}, where the projection $\pi$ sends a path $p$ to its value at $2\pi$. $PG$ is contractible and so the path fibration is a model for the universal $\Omega G$-bundle and we have $B\Omega G = G$. Given another $\Omega G$-bundle $P \to M$ we can write down a classifying map as follows \cite{Murray-Vozzo1}: Choose a Higgs field $\Phi$ for $P$. Then for $p \in P$ we define $\hol_\Phi (p)$ to be the solution to the equation $\Phi(p) = q^{-1}\partial q$ for $q \in PG.$ (Note that $q\in PG$ implies that $q(0)=1$ which ensures that the solution is unique.) The map $\hol_\Phi$ descends to a map (also called $\hol_\Phi$) $M \to G$. We call this map the \emph{Higgs field holonomy} for $P$ and it gives a classifying map for the $\Omega G$-bundle $P$.

A connection for the path fibration is given in \cite{Carey:2002}: Let $\a\colon \RR \to \RR$ be a smooth function such that $\a(t) = 0\ \forall t\leq 0 $ and $\a(t) = 1\ \forall t \geq 2\pi$. Then
$$
A = \Theta - \a \, \ad(p^{-1})  \pi^* \widehat{\Theta}
$$
defines a connection, where $\Theta$ is the (left invariant) Maurer-Cartan form on $G$ and $\widehat{\Theta}$ is the \emph{right} invariant Maurer-Cartan form. The curvature of this connection is
$$
F =  \frac{1}{2}\left(\a^2 -\a \right) \ad(p^{-1}) [ \pi^* \widehat{\Theta}, \pi^* \widehat{\Theta}].
$$
A Higgs field for $PG$ is given by
$$
\Phi(p) = p^{-1} \partial p.
$$
Its covariant derivative is
$$
\nabla\Phi = \partial \a \ad(p^{-1}) \pi^*\hat{\Theta}.
$$
We call these the  {\em standard} connection and Higgs field for the path fibration.

Using Theorem \ref{T:string class} and the standard connection and Higgs field for the path fibration we can calculate the universal string class for $\Omega G$-bundles $s(PG) \in H^3(G)$:
\begin{align*}
s(PG) &= -\frac{1}{4\pi^2} \int_{S^1} \left\< \frac{1}{2}\left(\a^2 -\a \right) \ad(p^{-1}) [ \pi^* \widehat{\Theta}, \pi^* \widehat{\Theta}], \partial \a \ad(p^{-1}) \pi^*\hat{\Theta} \right\> d\theta\\
	&= \frac{1}{48\pi^2} \left\< [\Theta, \Theta], \Theta\right \>
\end{align*}
which is the generator of the degree three cohomology of $G$.


\subsection{Higher string classes for $\Omega G$-bundles}\label{SS:higher Omega G classes}

The material above is actually the degree three case of a more general construction presented in \cite{Murray-Vozzo1}. In that paper we showed how to obtain higher degree analogues of the string class for any $\Omega G$-bundle using the \emph{caloron correspondence}, a correspondence between loop group bundles and certain $G$-bundles. We shall outline these results here before extending them to $LG$-bundles in section \ref{S:higher LG classes}. We begin with the caloron correspondence.

\subsubsection{The caloron correspondence}

The caloron correspondence gives a bijection between isomorphism classes of $LG$-bundles over a manifold $M$ and isomorphism classes of $G$-bundles over $M\times S^1.$ It was first introduced in \cite{Garland:1988} as a bijection between isomorphism classes of $G$-instantons on $\RR^3 \times S^1$ and $\Omega G$-monopoles on $\RR^3.$ The form we will use in this article, however is from \cite{Murray:2003} where it was used to relate the string class of an $LG$-bundle to the Pontrjagyn class of a $G$-bundle.

Let $\widetilde{P} \to M \times S^1$ be a $G$-bundle and consider the $LG$-bundle $L\wt P \to L(M \times S^1)$. Now pull this bundle back by the map $\eta\colon M \to L(M\times S^1)$ given by $\eta(m) = (\theta \mapsto (m, \theta)).$ This gives an $LG$-bundle $P \to M$. Furthermore, given a connection $\tilde{A}$ on $\wt P,$ we can define a connection $A$ on $P$ by first defining a connection $L\tilde{A}$ on $L\wt P$ given by acting pointwise (i.e. $L\tilde{A}_\gamma(X)(\theta) = \tilde{A}_{\gamma(\theta)}(X(\theta))$, for $X$ any tangent vector to $\gamma \in L\wt P$) then pulling back by $\eta,$ so $A = \eta^* L\tilde A$.

On the other hand, given an $LG$-bundle $P \to M$, we can define a $G$-bundle $\wt P$ by $\wt P = (P \times G \times S^1)/LG$ where the $LG$ action is given by $(p, g, \theta)\gamma = (p\gamma, \gamma(\theta)^{-1}g, \theta)$ and the $G$ action is multiplication in the second factor. In order to define a connection $\tilde A$ on $\wt P$ we need to choose a Higgs field $\Phi$ for $P$. Then, as a form on $P \times G \times S^1$ which descends to the quotient, we have
$$
\tilde{A} = \ad(g^{-1}) A (\theta) + \Theta + \ad(g^{-1})\Phi \, d\theta.
$$

Using these constructions we can obtain the result from \cite{Murray:2003}:

\begin{proposition}[Caloron correspondence \cite{Murray:2003}]\label{P:caloron correspondence}
The constructions above are inverses of one another and give a bijection between isomorphism classes of $G$-bundles with connection over $M \times S^1$ and isomorphism classes of $LG$-bundles with connection and Higgs field over $M$.
\end{proposition}

Here we will need a slight modification of this result, given in \cite{Murray-Vozzo1}. Specifically, in order to construct an $\Omega G$-bundle $P\to M$ given a $G$-bundle $\wt P \to M\times S^1$ we need a way of choosing a basepoint in each fibre of $P$. This is given by choosing a section of $\wt P$ over $M\times \{0\}.$ Such a section is called a \emph{framing} and $\wt P$ is called a \emph{framed} $G$-bundle. A connection $A$ on $\wt P$ is called framed if it is flat with respect to the framing. We have

\begin{proposition}[\cite{Murray-Vozzo1}]
There is a bijection between isomorphism classes of framed $G$-bundles with framed connection over $M\times S^1$ and isomorphism classes of $\Omega G$-bundles with connection and Higgs field over $M$.
\end{proposition}

In both the free loop and based loop cases we call the $G$-bundle $\wt P$ the \emph{caloron transform} of $P$.


\subsubsection{Higher string classes for $\Omega G$-bundles}\label{SSS:higher Omega G classes}

In \cite{Murray-Vozzo1} we used the caloron correspondence to define characteristic classes for $\Omega G$-bundles in all odd degrees by using the Chern-Weil map for the corresponding $G$-bundle and integrating over the circle:

\begin{definition}[\cite{Murray-Vozzo1}]\label{D:Omega G string class}
Let $P \to M$ be an $\Omega G$-bundle with connection $A$ and Higgs field $\Phi$ and $\wt P \to M \times S^1$ be the caloron transform of $P$ with connection $\tilde A$. Then  the \emph{string form} of $p \in I^k(\fg)$ is
$$
s_p(A, \Phi) = \int_{S^1} cw_p(\tilde A),
$$
where $I^k(\fg)$ is the set of all multilinear, symmetric, $\ad$-invariant maps $\fg \times \ldots \times \fg \to \RR$ and $cw \colon I^k(\fg) \to \Omega^{2k-1}(M \times S^1)$ is the Chern-Weil map for the $G$-bundle $\wt P$ (that is, $cw_p(\tilde A) = p(\tilde F, \ldots, \tilde F)$ for $\tilde F$ the curvature of $\tilde A$).
\end{definition}

A formula for the string form in terms of data on the $\Omega G$-bundle $P$ is given by
$$
s_p(A, \Phi) = \int_{S^1} p(\nabla \Phi, F,\ldots, F)\, d \theta,
$$
where $F$ is the curvature of $A$. This can be seen by calculating the curvature $\tilde F$ of the connection $\tilde{A} = \ad(g^{-1}) A (\theta) + \Theta + \ad(g^{-1})\Phi \, d\theta $ in terms of $F$ and $\Phi:$
$$
\tilde F = \ad(g^{-1})(F + \nabla \Phi \, d \theta)
$$
and substituting this into the formula for the string class in Definition \ref{D:Omega G string class}. From now on we shall write $p(\nabla \Phi, F^{k-1})$ for $p(\nabla \Phi, F,\ldots, F)$ and so on.

In \cite{Murray-Vozzo1} it is shown that $s_p(A, \Phi)$ is closed and independent of the choice of connection and Higgs field. We call the class of the form $s_p(A, \Phi)$ the \emph{string class} of $P$ (associated to $p$) and write $s_p(P) \in H^{2k-1}(M)$. The main result of \cite{Murray-Vozzo1} is then

\begin{theorem}[\cite{Murray-Vozzo1}]\label{T:"Chern-Weil"}
If $P \to M$ is an $\Omega G$-bundle with caloron transform $\widetilde{P} \to M\times S^1$ and 
$$
s(P) \colon I^k(\fg) \to H^{2k-1}(M)
$$
is the map which associates to any invariant polynomial $p$ the string class of $P$, then the following diagram commutes
$$
\xymatrix@R=9ex@C=15ex{ I^k(\fg) \ar^{s(P)}[dr]  \ar^{cw(\wt P)}[r]  \ar^{\tau}[d]	& H^{2k}(M \times S^1) \ar^{\int_{S^1}}[d]\\
	H^{2k-1}(G) \ar^{\hol_\Phi^*}[r]		& H^{2k-1}(M)}
$$
Here $\tau$ is the transgression map given by (see section \ref{SS:equivariant transgressions})
$$
\tau(p) = \left( -\frac{1}{2} \right)^{k-1} \frac{k! (k - 1)!}{(2k-1)!} \, p(\Theta, [\Theta, \Theta], \ldots, [\Theta, \Theta]),
$$
$cw(\widetilde{P})$ is the usual Chern-Weil homomorphism for the $G$-bundle $\widetilde{P}$ and $\hol_\Phi$ is the classifying map of $P$.
\end{theorem}


\section{Classifying theory of $LG$-bundles}\label{S:classifying LG}

\subsection{The universal bundle}

In order to extend the ideas from the previous section we need a model for $EL G.$ This will allow us to calculate the universal string class. To construct this we view $L G$ as the semi-direct product $\Omega G\rtimes G.$ The group multiplication is given by
$$
(\gamma_1, g_1) (\gamma_2, g_2) = (g_2^{-1} \gamma_1 g_2 \gamma_2 , g_1g_2)
$$
and the isomorphism between $\Omega G \rtimes G$ and $L G$ is
\begin{align*}
	\Omega G \rtimes G &\xrightarrow{\sim} L G;\quad
	(\gamma, g)		\mapsto g\gamma.
\end{align*}
On the level of Lie algebras, the isomorphism is
\begin{align*}
	\Omega \fg \rtimes \fg &\xrightarrow{\sim} L \fg;\quad
	(\xi, X)		\mapsto X + \xi.
\end{align*}
We therefore need a model for the universal $\Omega G \rtimes G$-bundle. For this, we shall take the product of the universal $\Omega G$-bundle and the universal $G$-bundle. Recall that a model for the universal $\Omega G$-bundle is given by the space of paths $p\colon \RR \to G$ such that $p(0) = 1$ and $p^{-1} \partial p$ is periodic. So, for our model for $EL G$ we shall take the space $P G \times EG$ which is contractible since $P G$ and $EG$ are both contractible. This is acted on by $\Omega G\rtimes G:$
$$
(p, x) (\gamma, g) = (g^{-1} p g \gamma, xg)
$$
where $xg$ is the right action of $G$ on $EG.$ This action is free (since $G$ acts on $EG$ freely) and transitive on fibres (since the action on $EG$ is transitive and the equation $g^{-1}p_1 g\gamma = p_2$ can always be solved) and so $P G \times EG$ is a model for $EL G$ and $BL G$ is equal to $(P G\times EG)/(\Omega G \rtimes G).$ In fact, if we consider the map
$$
(P G\times EG)/(\Omega G \rtimes G) \to (G\times EG)/G; \quad [p, x] \mapsto [p(2\pi), x],
$$
where $[h, x] = [g^{-1}hg, xg],$ we can see this is well-defined, since
$$
[p, x] = [g^{-1}pg\gamma, xg] \mapsto [g^{-1}p(2\pi)g \gamma(2\pi), xg] = [p(2\pi), x].
$$
Furthermore, this is surjective, as the projection $P G \to G$ is surjective, and injective, for if we consider two elements $[p,x],\,[q,y] \in(P G\times EG)/(\Omega G \rtimes G)$ such that $[p(2\pi), x] = [q(2\pi), y]$ we have $y = xg$ and $q(2\pi) = g^{-1}p(2\pi)g.$ That is, the paths $q$ and $g^{-1}p g$ have the same endpoint. Therefore, the path $g^{-1}p^{-1}gq$ is actually a (based) loop. And since $q = g^{-1}pg (g^{-1}p^{-1}gq),$ we have
\begin{align*}
[q, y]	&= [ g^{-1}pg \gamma, xg]\\
	&= [p, x],
\end{align*}
where $\gamma=  g^{-1}p^{-1}gq \in \Omega G.$ Thus we have a diffeomorphism between $BL G$ and $(G\times EG)/G$ (or simply $G\times_G EG$). Importantly for our purposes this allows us to identify the cohomology of $BL G$ with the equivariant cohomology of $G$ (with its adjoint action). That is,
$$
H(BL G) = H_G (G).
$$
We will review equivariant cohomology in section \ref{S:equivariant cohomology} as we will use it in the sequel to calculate the string classes for an $LG$-bundle.

\subsection{Classifying maps}\label{SS:classifying maps}

Given an $L G$-bundle $P\to M$ we can write down the classifying map of this bundle as follows. Choose a Higgs field, $\Phi,$ for $P.$ Then define the map $f \colon P \to P G \times EG$ by 
$$
f(q) = (\hol_\Phi (q) , f_G (q)),
$$
where $\hol_\Phi$ is the Higgs field holonomy and $f_G$ is the classifying map for the $G$-bundle associated to $P$ by the projection $L G \to G$ given by mapping a loop to its start/endpoint (or equivalently, the projection $\Omega G\rtimes G \to G$). That is, $f(q) = (p, x)$ where $p^{-1} \partial p  = \Phi(q) $ and $ x $ is $f_G$ applied to the image of $q$ in $P\times_{L G}G.$ It is easy to see that this is equivariant with respect to the $L G$ action and hence descends to a map $M\to BL G$ since if $(\gamma, g) \in \Omega G \rtimes G$ then
\begin{align*}
f(q(g\gamma))	&= (\hol_\Phi (q(g\gamma)), f_G(q)g)
\end{align*}
and so $f$ is equivariant in the $EG$ slot (by virtue of the fact that $f_G$ is a classifying map) and also in the $P G$ slot since if $\hol_\Phi(q) = p$ then
\begin{align*}
\Phi(q(g\gamma)) 	&= \ad((g\gamma)^{-1})\Phi(q) + (g\gamma)^{-1}\partial (g\gamma)\\
				&= \ad((g\gamma)^{-1})\Phi(q) + \gamma^{-1}\partial \gamma
\end{align*}
and
\begin{align*}
(p(\gamma, g))^{-1} \partial (p(\gamma, g)) &= (g^{-1}p g\gamma)^{-1}\partial(g^{-1}p g\gamma)\\
			&= \gamma^{-1}g^{-1}p^{-1}g (g^{-1} \partial p g \gamma + g^{-1}p g \partial \gamma)\\
			&= \ad((g\gamma)^{-1}) p^{-1}\partial p + \gamma^{-1}\partial \gamma
\end{align*}
and so $\hol_\Phi (q(g\gamma)) = p (\gamma, g) = \hol_\Phi (q)(g\gamma).$

\begin{note}
Much of the construction above can be readily generalised to any group which is a semi-direct product. This is detailed in \cite[Appendix B]{Vozzo:PhD}.
\end{note}

\section{Equivariant cohomology}\label{S:equivariant cohomology}

The results from the previous section imply that any generalisation of Theorem \ref{T:"Chern-Weil"} to the free loop group will necessarily involve equivariant cohomology. In particular, in the next section we shall calculate the universal string class for $LG$-bundles (in analogy with the corresponding calculation for the universal $\Omega G$-bundle in section \ref{SS:s(PG)}). This will be most conveniently represented as an equivariant differential form---an element of the so-called Cartan model of equivariant cohomology. In this section, therefore, we wish to review this theory. We shall mainly follow \cite{Guillemin-Sternberg} (see also \cite{Atiyah:1984}).

\subsection{The Borel model and the Weil model}\label{SS:Borel}

Let $X$ be a manifold with an action of the Lie group $G$. If this action is free (for example, if $X$ is the total space of a principal $G$-bundle) then $X/G$ is a manifold and the \emph{equivariant cohomology} of $X$ is given by the cohomology of the quotient: $H_G(X) = H(X/G).$ If the action is not free however, then the quotient may not be a manifold and the cohomology of $X/G$ may not be the correct object to study. In general the equivariant cohomology of $X$ is defined by $H_G(X) := H(X \times_G E),$
where $E$ is some contractible space on which $G$ acts freely. A convenient example of such a space is of course the total space of the universal bundle, $EG$. So,
$$
H_G(X) = H(X \times_G EG).
$$
This is called the \emph{Borel model} for equivariant cohomology. The difficulty here is that in general it is not easy to study forms on $X \times_G EG$. In order to circumvent this we will introduce an algebraic version of the Borel model---the Weil model. Let us first set some conventions for the action of the Lie algebra $\fg$ on $\Omega(X)$.

Since $G$ acts on $X$, it acts on $\Omega(X)$ too. By convention, if $\phi\colon G \to \Diff(X)$, we define the action of $G$ on $\Omega (X)$ to be pull-back by the \emph{inverse} of $\phi.$ That is, $g \cdot \omega = (\phi_g^{-1})^* \omega$. This is because $\phi$ satisfies $\phi_{gh} = \phi_g \phi_h$ whereas we would like to have a \emph{right} action of $G$ on $X$ (and $\Omega(X)$). Note that this means that the fundamental vector field generated by the Lie algebra element $\chi$ is given by
$$
\bar\chi_x := \left.\frac{d}{dt}\right|_0 \exp(-t\chi) \cdot x,
$$
for $x \in X$. We shall write $L_\chi$ and $\iota_\chi$ for the Lie derivative and contraction with this vector field, respectively. Note that if $G$ acts freely on $X$ then a form $\omega \in \Omega(X)$ descends to the quotient $X/G$ precisely if it is invariant and horizontal with respect to the $G$ action. That is, if $L_\chi \omega$ and $\iota_\chi \omega$ both vanish for all $\chi \in \fg$. Alternatively, if $\{\xi_i \}$ is a basis for $\fg$ then $\omega$ descends if $L_{\xi_i} \omega = \iota_{\xi_i}\omega = 0, \ i=1, \ldots, n.$ From now on we shall write $L_i$ and $\iota_i$ for $L_{\xi_i}$ and $\iota_{\xi_i}$ respectively. 

In order to rephrase the Borel model of equivariant cohomology in algebraic terms we shall use Weil's version of Chern-Weil theory, which first appeared in \cite{Cartan:1951}.

Define the Weil algebra to be the tensor product of the exterior algebra of $\fg^*$, the dual of $\fg$, with the symmetric algebra of $\fg^*$. That is, the graded algebra
$$
W := \wedge(\fg^*) \otimes S(\fg^*).
$$
Here the symmetric algebra is understood to be \emph{evenly} graded, so every pure element is assigned twice its usual degree. It will be convenient to write down generators for $W$. Let $\{\theta^i\}$ be a basis for $\fg^*$. Then $W$ is generated by the elements $\theta^i$ of degree 1 in the exterior algebra and the corresponding elements, which we shall call $\mu^i$, of degree 2 in the symmetric algebra. The Weil algebra is in fact a \emph{differential} graded algebra with the differential $d$ given (on generators) by 
\begin{align*}
d\theta^i &= \mu^i - \tfrac12 c_{jk}^i \theta^j \theta^k,\\
d\mu^i &= c_{jk}^i \mu^j \theta^k.
\end{align*}
So we can consider the cohomology of $W$ with respect to this differential. Here we are of course interpreting the $\theta^i$'s as left invariant forms on $G$. Indeed, suppose that $a$ is a connection on $EG$ and $f$ is its curvature (for connections on universal bundles see for example \cite{Narasimhan:1961, Narasimhan:1963, Schlafly:1980} or \cite{Dupont:1978} for the simplicial point of view). Then if we expand $a$ in terms of a basis for $\fg$, so $a = a^i \xi_i$, the $a^i$'s are 1-forms on $EG$ and we can interpret them as elements in $\wedge (\fg^*)$ by their values on fundamental vector fields. That is, if $\chi \in \fg$ generates the vector field $\bar\chi$ on $EG$ then $a^i(\bar\chi) = \chi^i.$ On the other hand, given an element $p\in S^k(\fg^*)$, evaluating it on $f = f^i \xi_i$ gives a polynomial in the $f^i$'s and so they generate the $S(\fg^*)$ part of $W$. In fact, it is easy to see \cite{Guillemin-Sternberg} that this is the horizontal part of $W$ and so the basic subalgebra of $W$---the invariant part of this---is just the invariant, symmetric polynomials in $\fg^*$, whence we get the Chern-Weil homomorphism.\footnote{In order to see this, one proves that there is a morphism of differential graded algebras from $W$ to $\Omega (X)$ for any $G$-space $X$ which maps $W_{\text{basic}}$ into $\Omega(X)_{\text{basic}}$ and hence the basic cohomology of $W$ into the basic cohomology of $\Omega(X)$. If $X \to M$ is a principal $G$-bundle then $H_{\text{basic}}(\Omega(X)) = H(M)$ and so we get the Chern-Weil homomorphism. For details, see \cite{Guillemin-Sternberg} (or \cite{Cartan:1951} for the original exposition).} 


Recall that the de Rham complex $\Omega(X)$ forms a graded algebra itself. The standard result, then, is the following
\begin{theorem}[Equivariant de Rham Theorem]
The equivariant cohomology of $X$ is given by the basic cohomology (that is, the cohomology of the basic subcomplex) of $W \otimes \Omega(X):$
$$
H_G(X) = H_{\text{basic}}(W \otimes \Omega(X)).
$$
\end{theorem}
The basic cohomology of the differential graded algebra $W \otimes \Omega(X)$ is called the \emph{Weil model} of the equivariant cohomology of $X$.

The connection-curvature construction above helps translate between the Borel model and the Weil model. However, it is still quite difficult to perform computations with the Weil model. Therefore, we shall present another model in the next section---called the Cartan model---and describe the method for translating between this model and the Weil model: the Mathai-Quillen isomorphism.

\subsection{The Cartan model and the Mathai-Quillen isomorphism}\label{SS:Cartan}

In order to perform calculations with equivariant cohomology it is useful to find an automorphism of the Weil model which simplifies things significantly. Define the degree zero endomorphism $\gamma \in \End(W \otimes \Omega (X))$ by
$$
\gamma = \theta^i \otimes \iota_i.
$$
Clearly this is nilpotent and so the automorphism
$$
\phi := \exp \gamma
$$
is a finite sum. This automorphism is known as the \emph{Mathai-Quillen isomorphism}. The following fundamental observation (due to Mathai-Quillen \cite{Mathai:1986} and Kalkman \cite{Kalkman:PhD}) is proved in \cite{Guillemin-Sternberg}

\begin{theorem}\label{T:MQ}
The Mathai-Quillen isomorphism carries the horizontal subspace $(W \otimes \Omega (X))_{\text{hor}}$ into $W_\text{hor} \otimes \Omega (X)= S(\fg^*) \otimes \Omega(X)$ and hence the basic subcomplex $(W \otimes \Omega(X))_\text{basic}$ into the invariant elements of $S(\fg^*) \otimes \Omega (X)$, denoted $(S(\fg^*) \otimes \Omega (X))^G$. On this basic subcomplex the differential is transformed into $1\otimes d - \mu^i \otimes \iota_i.$
\end{theorem}

We make the following

\begin{definition}\label{D:Cartan}
The cohomology of the complex $\Omega_G(X) := (S(\fg^*) \otimes \Omega (X))^G$ with the differential above, $1\otimes d - \mu^i \otimes \iota_i,$ is called the \emph{Cartan model} for the equivariant cohomology of $X$. The elements of the space $\Omega_G(X)$ are called \emph{equivariant differential forms}.
\end{definition}

Theorem \ref{T:MQ} implies that the equivariant cohomology according to the Cartan model agrees with that of the Weil model. The elements of $\Omega_G(X)$ have a pleasant interpretation as form-valued polynomials on $G$. The invariance condition translates with this interpretation into equivariance of the polynomials. The differential on an element $\omega$ is then given by
$$
d(\omega(\chi)) - \iota_{\chi}(\omega(\chi)).
$$

\begin{remark}\label{R:MQ}
Suppose that $\omega$ is a degree $n$ basic element of the Weil algebra. Then $\omega$ is a sum of terms, each lying in a different part of the graded algebra
$$
\bigoplus_{i + 2j + k =n} \wedge^i(\fg^*) \otimes S^j(\fg^*) \otimes \Omega^k(X).
$$
Further,
\begin{align*}
\gamma\colon& \wedge^i(\fg^*) \otimes S^j(\fg^*) \otimes \Omega^k(X) \to \wedge^{i+1}(\fg^*) \otimes S^j(\fg^*) \otimes \Omega^{k-1}(X),\\
\gamma^2\colon& \wedge^i(\fg^*) \otimes S^j(\fg^*) \otimes \Omega^k(X) \to \wedge^{i+2}(\fg^*) \otimes S^j(\fg^*) \otimes \Omega^{k-2}(X)
\end{align*}
and so on. But since $\omega$ is basic, we know that $\phi (\omega) \in (S(\fg^*) \otimes \Omega (X))^G.$ That is, it has no components which lie in $\wedge(\fg^*)$. Therefore, since $\phi = 1 + \gamma + \tfrac12 \gamma^2 + \ldots$ and every application of $\gamma$ raises the degree of the exterior algebra by one, the effect of $\phi$ is essentially to ``pick out'' the components of $\omega$ which have no component in $\wedge(\fg^*).$ The end result of $\phi$ acting on the rest of the terms is that they all must cancel.
\end{remark}

\section{The universal string class}\label{S:universal string class}

Now that we have a model for the universal $LG$-bundle we would like to prove the analogue of Theorem \ref{T:"Chern-Weil"} for $LG$-bundles. Note that this will naturally involve equivariant cohomology since rather than the universal string class being equal to the transgression map, which takes values in $H(G) = H(B \Omega G)$, the universal string class for $LG$-bundles will be in $H(G \times_G EG) = H_G(G)$. Thus the first thing we need to do is calculate this universal string class. To illustrate the idea we will do this in the degree three case first and extend the result to higher degrees in the next section.

\subsection{The string class in the Borel model}\label{SS:string Borel}

Firstly we will use Theorem \ref{T:string class} to calculate a differential form representing the string class of $ELG = PG \times EG$. Note that this will give us a class in $H(G \times_G EG),$ the Borel model for the equivariant cohomology of $G.$ In order to use Theorem \ref{T:string class}, the first thing we need is a connection on $P G \times EG.$ Now, $P G$ is already a smooth manifold. In order to define a smooth structure and find a connection on $EG$ we use the results in \cite{Narasimhan:1961, Narasimhan:1963, Schlafly:1980}. As long as the dimension of the base of the $G$-bundle $P \to M$ is less than or equal to $n$ this gives a construction of a smooth bundle $EG_n \to BG_n$ with connection which is a model for the universal $G$-bundle. From now on we assume therefore that the dimension of the base of our $L G$-bundle is fixed (and less than or equal to $n$ for some $n$).

 To define a connection we need to know what a vertical vector looks like. Consider the vector in $T_{(p, x)}(P G\times EG_n) = T_p P G \times T_x EG_n$ generated by the Lie algebra element $(\xi, X) \in \Omega \fg \rtimes \fg:$
\begin{align*}
\iota_{(p,x)} (\xi, X) 	&=  \frac{d}{dt}\bigg|_0 ( (1-tX) p (1+tX)(1+t\xi), x e^{tX})\\
				&=  \frac{d}{dt}\bigg|_0 ( t(-Xp + pX + p\xi), x e^{tX})\\
				&= ( p (X - \ad(p^{-1})X + \xi), \iota_x (X)).
\end{align*}

A connection is given \cite{Vozzo:PhD} by
$$
A = \Theta - \a \ad(p^{-1})\left\{ \ev_{2\pi}^*\hat{\Theta} - \ad(p(2\pi))a + a \right\}  + \ad(p^{-1})a.
$$
As before, $\Theta$ is the Maurer-Cartan form, $\hat{\Theta}$ is the right Maurer-Cartan form and $\a$ is a smooth function such that $\a(t) =0 \ \forall t\leq 0$ and $\a(t)=1 \ \forall t\geq 2\pi.$ The function $\ev_{2\pi} \colon P G \to G$ is the projection and $a$ is a connection on $EG$ (which we shall assume we have using the results cited earlier). It can be easily checked that this form returns the Lie algebra element $\xi + X \in L\fg$ when evaluated on the vertical vector above and also that it transforms in the adjoint representation. Thus it satisfies the conditions for a connection.

To calculate the string class we will need the curvature of this connection and a Higgs field. As usual, the curvature is given by the formula
$$
F = DA
$$
where $D$ is the covariant exterior derivative. So we have
$$
F ((V,W), (V',W')) = -\tfrac{1}{2} A ([h(V,W), h(V',W')])
$$
for $hX$ the projection onto the horizontal subspace of the vector $X$.
This yields
\begin{multline*}
F =
 \left(\a^2 - \a\right) \ad(p^{-1}) \left\{
 \tfrac{1}{2}[\ev_{2\pi}^*\hat{\Theta}, \ev_{2\pi}^*\hat{\Theta}] 
 - [\ev_{2\pi}^*\hat{\Theta}, \ad(p(2\pi)^{-1})a] \right.\\
 \left.+ [\ev_{2\pi}^*\hat{\Theta}, a] 
+\tfrac12 \ad(p(2\pi)) [a,a]
 - [\ad(p(2\pi))a, a] 
 + \tfrac{1}{2} [a,a]\right\}  \\
 + \a \ad(p^{-1})(\ad(p(2\pi))f - f) 
 + \ad(p^{-1})f
\end{multline*}
where $f$ is the curvature of $a.$

The other piece of data we need to calculate the string class is a Higgs field for $EL G.$ Define the map $\Phi \colon P G \times EG_n \to L\fg$ by
$$
\Phi (p,x) = p^{-1} \partial p.
$$
Then by the calculation at the end of section \ref{SS:classifying maps} we see that $\Phi$ is a Higgs field for $P G \times EG_n.$ Next we need to calculate
$$
\nabla \Phi = d \Phi + [A, \Phi] - \partial A.
$$
We can show
$$
d \Phi = [\Phi, \Theta] + \partial \Theta
$$
and so
\begin{multline*}
\nabla \Phi = [\Phi, \Theta] + \partial \Theta\\ 
+ \left[ \Theta - \a \ad(p^{-1})\left\{ \ev_{2\pi}^*\hat{\Theta} - \ad(p(2\pi))a + a \right\}  + \ad(p^{-1})a, \Phi\right]\\
 - \partial \left( \Theta - \a \ad(p^{-1})\left\{ \ev_{2\pi}^*\hat{\Theta} - \ad(p(2\pi))a + a \right\}  + \ad(p^{-1})a \right)
\end{multline*}
\vspace{-3ex}
\begin{multline*}
\phantom{\nabla \Phi} = \partial\a \ad(p^{-1}) \left\{ \ev_{2\pi}^*\hat{\Theta} - \ad(p(2\pi))a + a\right\}.\\
\end{multline*}
Therefore, using the formula from Theorem \ref{T:string class}, the string class for $P G \times EG_n$ as a class in $H(G \times_G EG)$ is
\begin{multline*}
-\frac{1}{4\pi^2}\int_{S^1}\left\< \left( \a^2 -\a \right) \left(
 \tfrac{1}{2}[\ev_{2\pi}^*\hat{\Theta}, \ev_{2\pi}^*\hat{\Theta}] 
 - [\ev_{2\pi}^*\hat{\Theta}, \ad(p(2\pi)^{-1})a] \right.\right.\\
\left.+ [\ev_{2\pi}^*\hat{\Theta}, a] 
 +\tfrac12 \ad(p(2\pi)) [a,a]
 - [\ad(p(2\pi))a, a] 
+ \tfrac{1}{2} [a,a]\right)\\
+ \a \ad(p^{-1})(ad(p(2\pi))f - f) 
 + \ad(p^{-1})f,\\
  \left. \partial \a \left( \ev_{2\pi}^*\hat{\Theta}
 - \ad(p(2\pi))a 
 + a\right)\right\>
\end{multline*}
\begin{multline*}
= - \frac{1}{8\pi^2}
\left\<-\tfrac16 [\hat{\Theta}, \hat{\Theta}]
+\tfrac13 [\hat{\Theta}, \ad(p(2\pi))a]
-\tfrac13 [\hat{\Theta}, a]\right.\\
+ \tfrac16 \ad(p(2\pi)) [a,a]
+\tfrac13 [\ad (p(2\pi))a , a]
- \tfrac16 [a,a]
+ \ad p(2\pi) f + f,\\
\left.\hat{\Theta} - \ad p(2\pi) a + a \right\>
\end{multline*}
\begin{multline*}
= \frac{1}{8\pi^2} \left\{
\tfrac16\big\< [\hat{\Theta}, \hat{\Theta}], \hat{\Theta}\big\>
-\tfrac12 \big\< [\hat{\Theta}, \hat{\Theta}], \ad( g) a\big\>
+ \tfrac12 \big\< [\hat{\Theta}, \hat{\Theta}],  a\big\> \right.\\
+ \tfrac12 \big\< \hat{\Theta}, \ad(g)[a,a]\big\>
- \big\< \hat{\Theta}, [\ad(g)a,a]\big\>
+ \tfrac12 \big\< \hat{\Theta}, [a,a]\big\>\\
+ \tfrac12 \big\< a, \ad(g)[a,a]\big\>
- \tfrac12 \big\< \ad(g)a, [a,a]\big\>
- \big\< \ad(g) f + f, \hat{\Theta}\big\>\\
\left.\vphantom{\hat{\Theta}}- \big\< \ad(g) f - \ad(g^{-1})f, a\big\> \right\}
\end{multline*}
for $g = p(2\pi).$

\subsection{The string class in the Cartan model}\label{SS:string Cartan}

The formula above for the string class is quite unwieldy. We have already seen that the most compact representation for equivariant cohomology classes is via equivariant differential forms. We shall now proceed to find an equivariant differential form representing the same class as above. Firstly, let us write this form as an element of the Weil model. The degree three component of the Weil model is given by
\begin{multline*}
(W \otimes \Omega (G))^3 = \bigoplus_{i+2j+k = 3} \wedge^i(\fg^*) \otimes S^j(\fg^*) \otimes \Omega^k(G)\\
= (\wedge^0 \otimes S^0 \otimes \Omega^3) \oplus 
(\wedge^1 \otimes S^0 \otimes \Omega^2) \oplus 
(\wedge^2 \otimes S^0 \otimes \Omega^1) \oplus 
(\wedge^3 \otimes S^0 \otimes \Omega^0)\\ \oplus
(\wedge^0 \otimes S^1 \otimes \Omega^1) \oplus 
(\wedge^1 \otimes S^1 \otimes \Omega^0).
\end{multline*}
We can use the connection-curvature construction from section \ref{SS:Borel} to rewrite the string class as follows. Consider, for example, the terms $-\tfrac12 \big\< [\hat{\Theta}, \hat{\Theta}], \ad( g) a\big\>$ and $ \tfrac12 \big\< [\hat{\Theta}, \hat{\Theta}],  a\big\>$ above. In the Weil model, these terms should live in $\wedge^1 \otimes S^0 \otimes \Omega^2$. If we expand the connection in terms of a basis $\{\xi_i\}$ for $\fg$ then we can write
\begin{align*}
\tfrac12 \big\< [\hat{\Theta}, \hat{\Theta}],  a\big\> -\tfrac12 \big\< [\hat{\Theta}, \hat{\Theta}], \ad( g) a\big\> &=  \tfrac12 \big\< [\hat{\Theta}, \hat{\Theta}],  a^i\xi_i\big\> -\tfrac12 \big\< [\hat{\Theta}, \hat{\Theta}], \ad( g) a^i\xi_i\big\>\\
&= \tfrac12 a^i \big\< [\hat{\Theta}, \hat{\Theta}], \xi_i -\ad( g) \xi_i\big\>.
\end{align*}
Therefore the element in $\wedge^1 \otimes S^0 \otimes \Omega^2$ corresponding to these two terms is
$$
\tfrac12 a^i \big\< [\hat{\Theta}, \hat{\Theta}], \xi_i -\ad( g) \xi_i\big\>.
$$
Similarly, for the terms with the connection appearing twice (the middle three terms), we have
$$
\tfrac12 a^i \wedge a^j \big\< \hat\Theta , c_{ij}^k (\xi_k + \ad(g) \xi_k) \big\> - a^i \wedge a^j \big\< \hat\Theta , [\ad(g) \xi_i, \xi_j] \big\> \in \wedge^2 \otimes S^0 \otimes \Omega^1,
$$
where $c_{ij}^k$ are the structure constants for $\fg$. The last two terms live in $\wedge^0 \otimes S^1 \otimes \Omega^1$ and $\wedge^1 \otimes S^1 \otimes \Omega^0,$ respectively and we view them as the $\Omega^1$-valued ($\wedge^1$-valued) polynomials on $\fg:$
$$
\chi \mapsto - \< \ad(g) \chi + \chi, \hat\Theta\> 
\quad \text{and} \quad 
\chi \mapsto - a^i \<\ad(g) \chi - \ad(g^{-1}) \chi, \xi_i \>.
$$

We can now apply the Mathai-Quillen isomorphism to the string class $s$ (interpreted as above). That is, we can calculate\footnote{Note that in light of remark \ref{R:MQ} we can already identify the image of $s$ under $\phi$ since we know that it is a basic form (indeed we are viewing it as a form on $G \times EG$ which descends to $G \times_G EG$!). However, we shall present the calculation here as it illustrates the beauty of Mathai and Quillen's result and Cartan's formalism. } $\phi(s) = (\exp \gamma)(s) = s + \gamma(s) + \ldots$ We will calculate a few terms here and leave the rest as an exercise. Let us begin with the term $\tfrac16 \< [\hat\Theta, \hat\Theta], \hat\Theta\> =: s_{003}.$ Recalling that the $a^i$ are the connection elements in $W$, we have
\begin{align*}
\gamma (s_{003})& = \tfrac16 (a^i \otimes \iota_i) (\< [\hat\Theta, \hat\Theta], \hat\Theta\>)\\
&= \tfrac12 a^i \< [\iota_i\hat\Theta, \hat\Theta], \hat\Theta\>.
\end{align*}
Now, recall that here $\iota_i$ means contraction with the vector field generating the action of $G$ on itself---in this case the adjoint action. So $\iota_\chi \hat{\Theta}$ is given by
\begin{align*}
\hat\Theta_g \left( \left.\frac{d}{dt} \right|_0 \exp(-t \chi) g \exp( t\chi) \right)
 &= \hat\Theta_g \left( g \chi - \chi g \right)\\
 &= \ad(g) \chi - \chi.
\end{align*}
Therefore,
$$
\gamma(s_{003}) = \tfrac12 a^i \< [\hat\Theta, \hat\Theta], \ad(g) \xi_i - \xi_i \>.
$$
Notice that this precisely cancels out the term $s_{102} := s |_{\wedge^1\otimes S^0\otimes \Omega^2} =  \tfrac12 a^i \< [\hat\Theta, \hat\Theta], \xi_i - \ad(g)\xi_i \>$. We can calculate the $\wedge^2\otimes S^0 \otimes \Omega^1$ part of $\phi(s)$ using
\begin{align*}
s_{201} \phantom{)} &= \tfrac12 a^i \wedge a^j \big\< \hat\Theta , c_{ij}^k (\xi_k + \ad(g) \xi_k) \big\> - a^i \wedge a^j \big\< \hat\Theta , [\ad(g) \xi_i, \xi_j] \big\> ,\\
\gamma(s_{102}) &= - a^i \wedge a^j \big\< \hat\Theta , c_{ij}^k (\xi_k + \ad(g) \xi_k) \big\> + 2 a^i \wedge a^j \big\< \hat\Theta , [\ad(g) \xi_i, \xi_j] \big\> = - 2s_{201},\\
\tfrac12\gamma^2(s_{003}) &= \tfrac12 a^i \wedge a^j \big\< \hat\Theta , c_{ij}^k (\xi_k + \ad(g) \xi_k) \big\> - a^i \wedge a^j \big\< \hat\Theta , [\ad(g) \xi_i, \xi_j] \big\> = s_{201}.
\end{align*}
(Here we have written $\omega_{ijk}$ for the $\wedge^i\otimes S^j \otimes \Omega^k$ part of the element $\omega \in W\otimes \Omega(G)$.) We can calculate all the other terms similarly and we find that everything cancels (as indeed it must!) except for the terms in $\wedge^0 \otimes S^0 \otimes \Omega^3$ and $\wedge^0 \otimes S^1 \otimes \Omega^1$. These are given by $ \tfrac16 \< [\hat\Theta, \hat\Theta], \hat\Theta\>$ and $ \chi \mapsto -\< \chi, \Theta + \hat\Theta \>.$ Therefore, we find
\begin{theorem}\label{T:universal string}
The universal string class for $LG$-bundles is represented by the class of the equivariant differential form
$$
\frac{1}{8\pi^2} \left(\frac16 \< [\hat\Theta, \hat\Theta], \hat\Theta\> - \< \chi, \Theta + \hat\Theta \> \right).
$$
\end{theorem}

\begin{remark}
The equivariant form above coincides with the equivariant extension of the transgression form $\frac{1}{48\pi^2} \< \hat\Theta, [\hat\Theta, \hat\Theta] \>$ defined by Alekseev and Meinrenken \cite{Alekseev:2009}. There are also equivariant extensions of the higher analogues of these forms and in the next section we shall show that these represent the universal string classes in all odd dimensions.
\end{remark}

\section{Higher string classes for $LG$-bundles}\label{S:higher LG classes}

We would now like to extend Theorem \ref{T:universal string} to string classes in all odd degrees (in analogy with the results for $\Omega G$-bundles from \cite{Murray-Vozzo1}). As remarked at the end of the last section we wish to show that these coincide with the classes of the equivariant transgression forms as defined in \cite{Alekseev:2009} and \cite{Jeffrey:1995}. Therefore, let us first review the construction of these forms.

\subsection{Equivariant transgression forms}\label{SS:equivariant transgressions}

\subsubsection{Transgression forms}

%

Let $p\in I^k(\fg)$ and define the 1-form on $G \times [0,1]$ by $t \Theta$. Define the `curvature' $F_t$ of this form by $d(t\Theta) + \frac12 [t\Theta, t\Theta]$ and consider the $(2k-1)$-form on $G$ given by
\begin{align*}
- \int_0^1 p(F_t^k) &= k\int_0^1 p(\Theta, (\tfrac12 (t^2 - t) [\Theta, \Theta])^{k-1}) \, dt\\
	&= \left( -\frac{1}{2} \right)^{k-1} \frac{k! (k - 1)!}{(2k-1)!} \, p(\Theta, [\Theta, \Theta]^{k-1}).
\end{align*}
It is easy to see that this form is closed and we call it the \emph{transgression} of $p$. We write $\tau(p)$ (for both the $(2k-1)$-form and the class in $H^{2k-1}(G)$.)

\subsubsection{Equivariant transgression forms}

To define the equivariant version of the form above, consider the `equivariant curvature' of the 1-form above, $F_G(t\Theta)(\chi),$ given by $d_G(t\Theta)(\chi) + \frac12 [t\Theta, t\Theta]$ where $d_G = d - \iota_\chi$ is the equivariant differential from the Cartan model (Definition \ref{D:Cartan}). The \emph{equivariant transgression} of $p$ is given by \cite{Alekseev:2009, Jeffrey:1995}
$$
\tau_G(p) = -\int_0^1 p((F_G(t\Theta)(\chi) + \chi)^k).
$$
We have 
$$
F_G(t\Theta)= dt \wedge \Theta + \tfrac12 (t^2 - t) [\Theta, \Theta] - t(\chi -\ad (g^{-1})\chi)
$$
and so
\begin{align*}
\tau_G(p) &= -\int_0^1 p((dt \wedge \Theta + \tfrac12 (t^2 - t) [\Theta, \Theta] - t(\chi -\ad (g^{-1})\chi) + \chi)^k)\\
	&= k \int_0^1 p (\Theta, (\tfrac12 (t^2-t)[\Theta, \Theta] + (1-t)\chi +  t \ad(g^{-1})\chi)^{k-1}) \, dt.
\end{align*}
It is easy to see that $\tau_G(p)$ is equivariantly closed and so defines a class in $H^{2k-1}_G(G)$.

\subsection{Higher string classes for $LG$-bundles}\label{SS:higher classes}

Let $P\to M$ be an $LG$-bundle with connection $A$  and Higgs field $\Phi$ and let $F$ be the curvature of $A$ and $\nabla \Phi = d\Phi + [A, \Phi] - \partial A$ the covariant derivative of $\Phi$. Following \cite{Murray-Vozzo1} we define the \emph{string form} of the pair $(A, \Phi)$ associated to $p \in I^k(\fg)$ by
$$
s_p(A, \Phi) = \int_{S^1} cw_p(\tilde A),
$$
for $\tilde A$ the connection on the caloron transform of $P$. As in section \ref{SSS:higher Omega G classes} we have the following formula for the string form:
$$
s_p(A,\Phi) = k \int_{S^1} p(\nabla \Phi, F^{k-1}) \, d\theta.
$$
We also have the results analogous to \cite{Murray-Vozzo1}:

\begin{proposition}\label{P:closed}
The string form is closed and so defines a cohomology class in $H^{2k-1}(M).$
\end{proposition}

This class, which we denote by $s_p(P),$ is called the \emph{string class} of $P$ associated to $p$. 

\begin{proposition}\label{P:indep}
The string class is independent of the choice of connection and Higgs field.
\end{proposition}

\begin{proposition}\label{P:natural}
The string class is natural. That is, for $f: N\to M$ we have $s_p(f^*P) = f^*s_p(P)$. In particular, $s_p$ defines a characteristic class for $LG$-bundles.
\end{proposition}

The proofs of Propositions \ref{P:closed}, \ref{P:indep} and \ref{P:natural} are all identical to the $\Omega G$ case.

As in section \ref{S:universal string class} the string class associated to $p \in I^k(G)$ of the universal $LG$-bundle is an equivariant cohomology class. We have the following result concerning this class.

\begin{proposition}\label{P:s(ELG)}
The string class associated to $p \in I^k(\fg)$ of the universal $LG$-bundle is represented in $H^{2k-1}_G(G)$ by the equivariant transgression of $p.$ That is,
$$
s_p(ELG) = \tau_G(p).
$$
\end{proposition}

\begin{proof}
Recall that a connection and Higgs field for the universal bundle were given in section \ref{SS:string Borel}. The curvature of this connection was
\begin{multline*}
F =
 \left(\a^2 - \a\right) \ad(p^{-1}) \left\{
 \tfrac{1}{2}[\ev_{2\pi}^*\hat{\Theta}, \ev_{2\pi}^*\hat{\Theta}] 
 - [\ev_{2\pi}^*\hat{\Theta}, \ad(p(2\pi)^{-1})a] \right.\\
 \left.+ [\ev_{2\pi}^*\hat{\Theta}, a] 
+ \tfrac12 \ad(p(2\pi)) [a,a]
 - [\ad(p(2\pi))a, a] 
 + \tfrac{1}{2} [a,a]\right\}  \\
 + \a \ad(p^{-1})(\ad(p(2\pi))f - f) 
 + \ad(p^{-1})f
\end{multline*}
where $a$ is a connection on $EG$ and $f$ is its curvature. The covariant derivative of the Higgs field was given by
$$
\nabla \Phi = \partial\a \ad(p^{-1}) \left\{ \ev_{2\pi}^*\hat{\Theta} - \ad(p(2\pi))a + a\right\}.
$$

We wish to calculate the equivariant differential form representing the string class
$$
s_p = k\int_{S^1} p(\nabla \Phi, F^{k-1})\, d\theta
$$
(we omit the dependence on the $LG$-bundle since this is the universal class). Recall that in order to do this we use the Mathai-Quillen isomorphism which, according to Remark \ref{R:MQ}, requires us to ignore any terms in $s_p$ which involve the connection $a$. Therefore we see that the Mathai-Quillen isomorphism applied to the string class gives the equivariant form
\begin{align*}
s_p &= k \int_{S^1} p(\hat{\Theta}, (\tfrac12 (\a^2 - \a)[\hat\Theta, \hat\Theta] + \a (\ad(g) \chi - \chi) +\chi )^{k-1}) \, \partial \a \, d\theta\\
	&= k \int_0^1 p(\Theta, (\tfrac12 (\a^2 - \a)[\Theta,\Theta]  + \a ( \chi - \ad(g^{-1})\chi) + \ad(g^{-1})\chi )^{k-1})\, d\a\\
	&= k \int_0^1 p(\Theta, (\tfrac12 (\a^2 - \a)[\Theta,\Theta]  + \a  \chi +(1-\a) \ad(g^{-1})\chi )^{k-1})\, d\a.
\end{align*}
Now, making the change of variables $\a \mapsto 1-t$ gives the required result.

\end{proof}

We now have the result we were looking for (in analogy with Theorem \ref{T:"Chern-Weil"})

\begin{theorem}\label{T:LG Chern-Weil}
Let $P \to M$ be an $LG$-bundle and $\wt P \to M\times S^1$ its caloron transform. If
$$
s(P) \colon I^k(\fg) \to H^{2k-1}(M)
$$
is the map which gives for any $p \in I^k(\fg)$ its associated string class, then the following diagram commutes
$$
\xymatrix@R=9ex@C=15ex{ I^k(\fg) \ar^{s(P)}[dr]  \ar^{cw(\wt P)}[r]  \ar^{\tau_G}[d]	& H^{2k}(M \times S^1) \ar^{\int_{S^1}}[d]\\
	H^{2k-1}_G(G) \ar^{f^*}[r]		& H^{2k-1}(M)}
$$
Here $f^*$ is the map induced on cohomology by the classifying map $f$ of $P$ (section \ref{SS:classifying maps}).
\end{theorem}

\begin{proof}
The upper triangle in the diagram commutes by the definition of $s(P).$ Proposition \ref{P:natural} tells us that the string class is natural, so for the lower triangle it is enough to calculate the universal string class and then $s_p(P) = f^*s_p(ELG)$. But we know from Proposition \ref{P:s(ELG)} that $s_p(ELG)$ is equal to the equivariant transgression $\tau_G(p).$ Therefore, the diagram commutes.

\end{proof}


\begin{thebibliography}{99}

\bibitem{Alekseev:2009} A. Alekseev and E. Meinrenken. The Atiyah algebroid of the path fibration over a Lie group. \emph{Lett. Math. Phys.}, 90:23--58, 2009. 

\bibitem{Atiyah:1984} M. F. Atiyah and R. Bott. The moment map and equivariant cohomology. \emph{Topology}, 23(1):1--28, 1984. 

\bibitem{Carey:2002} A. L. Carey and J. Mickelsson. The universal gerbe, Dixmier-Douady class, and gauge theory. \emph{Lett. Math. Phys.}, 59(1):47--60, 2002. 

\bibitem{Cartan:1951} H. Cartan. Notions d'alg\`ebre diff\'erentielle; application aux groupes de Lie et aux vari\'et\'es o\`u op\`ere un groupe de Lie. In \emph{Colloque de topologie (espaces fibres), Bruxelles, 1950}, pages 15--27. 

\bibitem{Dupont:1978} J. L. Dupont. \emph{Curvature and characteristic classes}. Lecture Notes in Mathematics, Vol. 640. Springer-Verlag, Berlin, 1978. 

\bibitem{Garland:1988} H. Garland and M. K. Murray. Kac-Moody monopoles and periodic instantons. \emph{Comm. Math. Phys.}, 120(2):335--351, 1988.

\bibitem{Guillemin-Sternberg} V. W. Guillemin and S. Sternberg. \emph{Supersymmetry and equivariant de Rham theory}. Mathematics Past and Present. Springer-Verlag, Berlin, 1999. 

\bibitem{Jeffrey:1995} L. C. Jeffrey. Group cohomology construction of the cohomology of moduli spaces of flat connections on 2-manifolds. \emph{Duke Math. J.}, 77(2):407--429, 1995. 

\bibitem{Kalkman:PhD} J. Kalkman. \emph{A BRST model applied to symplectic geometry}. PhD thesis, Utrecht, 1993.

\bibitem{Killingback:1987} T. P. Killingback. World-sheet anomalies and loop geometry. \emph{Nuclear Phys. B}, 288(3-4):578--588, 1987. 

\bibitem{Mathai:1986} V. Mathai and D. Quillen. Superconnections, Thom classes, and equivariant differential forms. \emph{Topology}, 25(1):85--110, 1986. 

\bibitem{Murray:2003} M. K. Murray and D. Stevenson. Higgs fields, bundle gerbes and string structures. \emph{Comm. Math. Phys.}, 243(3):541--555, 2003. 

\bibitem{Murray-Vozzo1} M. K. Murray and R. F. Vozzo. The caloron correspondence and higher string classes for loop groups. \emph{J. Geom. Phys.}, 60(9):1235--1250 2010. 

\bibitem{Murray-Vozzo2} M. K. Murray and R. F. Vozzo. Circle actions, central extensions and string structures. \textsf{arXiv:1004.0779}, 2010. 

\bibitem{Narasimhan:1961} M. S. Narasimhan and S. Ramanan. Existence of universal connections. \emph{Amer. J. Math.}, 83:563--572, 1961. 

\bibitem{Narasimhan:1963} M. S. Narasimhan and S. Ramanan. Existence of universal connections II. \emph{Amer. J. Math.}, 85:223--231, 1963. 

\bibitem{Pressley-Segal} A. Pressley and G. Segal. \emph{Loop groups}. Oxford Mathematical Monographs. Oxford University Press, New York, 1986. 

\bibitem{Schlafly:1980} R. Schlafly. Universal connections. \emph{Invent. Math.}, 59(1):59--65, 1980. 

\bibitem{Vozzo:PhD} R. F. Vozzo. \emph{Loop groups, Higgs fields and generalised string classes}. PhD thesis, School of Mathematical Sciences, University of Adelaide, 2009. \textsf{arXiv:0906.4843}.


\end{thebibliography}

\end{document}